\documentclass[a4paper,10pt]{amsart}
\usepackage{amsmath,amssymb,color,multirow}
\usepackage[colorlinks]{hyperref}
\definecolor{linkblue}{RGB}{1,1,190}
\definecolor{citered}{RGB}{190,1,1}
\hypersetup{linkcolor=linkblue,urlcolor=linkblue,citecolor=citered}

\theoremstyle{plain}
\newtheorem{theorem}{\bf Theorem}[section]
\newtheorem{proposition}[theorem]{\bf Proposition}
\newtheorem{lemma}[theorem]{\bf Lemma}
\newtheorem{corollary}[theorem]{\bf Corollary}

\theoremstyle{definition}
\newtheorem{example}[theorem]{\bf Example}

\hypersetup{hypertexnames=false}
\numberwithin{equation}{section}

\makeatletter\@namedef{subjclassname@2020}{\textup{2020} Mathematics Subject Classification}\makeatother

\begin{document}
\title[On the system of length sets of power monoids]{On the system of length sets of power monoids}
\author{Andreas Reinhart}
\address{Institut f\"ur Mathematik und Wissenschaftliches Rechnen, Karl-Franzens-Universit\"at Graz, NAWI Graz, Heinrichstra{\ss}e 36, 8010 Graz, Austria}
\email{andreas.reinhart@uni-graz.at}
\keywords{fully elastic, length set, power monoid}
\subjclass[2020]{11B13, 11B30, 20M13}
\thanks{This work was supported by the Austrian Science Fund FWF, Project Number P36852}

\begin{abstract}
The set $\mathcal{P}_{{\rm fin},0}(\mathbb{N}_0)$ of all finite subsets of $\mathbb{N}_0$ containing the zero element is a monoid with set addition as operation. If a set $A\in\mathcal{P}_{{\rm fin},0}(\mathbb{N}_0)$ can be written in the form $A=\sum_{i=1}^{\ell} A_i$ with $\ell\in\mathbb{N}_0$ and indecomposable elements $(A_i)_{i=1}^{\ell}$ of $\mathcal{P}_{{\rm fin},0}(\mathbb{N}_0)$, then $\ell$ is a factorization length of $A$ and $\mathsf{L}(A)\subseteq\mathbb{N}_0$ denotes the set of all possible factorization lengths of $A$. We show that for each rational number $q\geq 1$, there is some $A\in\mathcal{P}_{{\rm fin},0}(\mathbb{N}_0)$ such that $q=\frac{\max(\mathsf{L}(A))}{\min(\mathsf{L}(A))}$. This supports a conjecture of Fan and Tringali.
\end{abstract}

\maketitle

\section{Introduction}\label{Section 1}

Let $M$ be a submonoid of an additive abelian group. Set addition of nonempty finite subsets of $M$ is a classic topic in additive combinatorics with innumerable facets. The sets $\mathcal{P}_{{\rm fin}}(M)$ resp. $\mathcal{P}_{{\rm fin},0}(M)$ of all nonempty finite subsets of $M$ (resp. of all finite subsets of $M$ containing the zero element) are additive monoids, called the (finitary resp. reduced finitary) power monoid of $M$, with set addition as operation and the set $\{0\}$ being its zero element. Suppose for the rest of this paragraph that $x+y=0$ implies $x=y=0$ for all $x,y\in M$. A nonempty finite subset $A\subseteq M$ is said to be {\it indecomposable} if any equation of the form $A=B+C$, with $B,C\subseteq M$, implies that $B=\{0\}$ or $C=\{0\}$. Thus, the indecomposable sets are precisely the atoms (irreducible elements) of the respective power monoids. The interplay between set addition problems of subsets of $M$ and properties of associated power monoids was first studied by Fan and Tringali in their pivotal paper \cite{FaTr18}.

\medskip
Algebraic and arithmetic properties of power monoids were first studied by Tamura and Shafer \cite{TaSh67} and found a strong renewed interest in the last decade (for a sample of recent work, see \cite{AnTr21,BiGe25,GaTr25,Ra26,Tr25,TrYa25a,TrYa25b}). Factorization theory describes the nonuniqueness of factorizations of elements in monoids and domains and relates arithmetic invariants with algebraic invariants of the objects under consideration. Length sets are the best investigated arithmetic invariants. To recall definitions, let $\mathcal{H}$ be a (suitable) additive monoid (such as the power monoid $\mathcal{P}_{{\rm fin}}(M)$) and let $A\in\mathcal{H}$. If $A=\sum_{i=1}^{\ell} A_i$ with $\ell\in\mathbb{N}_0$ and irreducible elements $(A_i)_{i=1}^{\ell}$ of $\mathcal{H}$, then $\ell$ is called a {\it factorization length} and the set $\mathsf{L}(A)\subseteq\mathbb{N}_0$ of all possible factorization lengths is called the {\it length set} of $A$. Then $\mathcal{L}(\mathcal{H})=\{\mathsf{L}(A):A\in\mathcal{H},\mathsf{L}(A)\not=\emptyset\}$ is the {\it system of length sets} of $\mathcal{H}$. For large classes of monoids (including Krull domains with finite class group \cite{GeSc23}) the length sets in $\mathcal{L}(\mathcal{H})$ are highly structured. Roughly speaking, they are generalized arithmetic progressions with global bounds on all parameters (for an overview, see \cite{GeHK06,Sc16}). On the other hand, there are monoids and domains for which every nonempty subset (resp. every nonempty finite subset of $\mathbb{N}_{\geq 2}$) occurs as a length set (see \cite{FaWi24,GeGo25,GeKa25,Go19} for recent progress in this direction).

\medskip
All length sets of $\mathcal{P}_{{\rm fin}}(\mathbb{N}_0)$ and of $\mathcal{P}_{{\rm fin},0}(\mathbb{N}_0)$ are finite, and a conjecture of Fan and Tringali \cite[Section 5]{FaTr18} states that these power monoids have the property that every nonempty finite subset of $\mathbb{N}_{\geq 2}$ occurs as a length set (their conjecture as well as our results deal with more general monoids $M$ and with various classes of power monoids but, for simplicity, we restrict the discussion to the power monoid $\mathcal{P}_{{\rm fin},0}(\mathbb{N}_0)$). The Fan--Tringali Conjecture is in contrast to a result of Bienvenu and Geroldinger \cite[Theorem 6.1]{BiGe25}, which says that almost all (in a natural sense of density) elements of $\mathcal{P}_{{\rm fin},0}(\mathbb{N}_0)$ are atoms (note that for the length set $\mathsf{L}(A)$ of an atom $A$ we have $\mathsf{L}(A)=\{1\}$). This indicates that progress on the Fan--Tringali Conjecture is not easy to get. Nevertheless, there are results supporting the conjecture: every positive integer $d\in\mathbb{N}$ occurs as a distance of a length set of $\mathcal{P}_{{\rm fin},0}(\mathbb{N}_0)$ and, for every $k\geq 2$, the union of all length sets containing $k$ equals $\mathbb{N}_{\geq 2}$ \cite[Theorem 4.11]{FaTr18}.

\medskip
The goal of the present paper is to make further progress towards the Fan--Tringali Conjecture by determining the set of elasticities of $\mathcal{P}_{{\rm fin},0}(\mathbb{N}_0)$. For a nonempty set $L\subseteq\mathbb{N}$, $\rho(L)=\frac{\sup(L)}{\min(L)}$ is the {\it elasticity} of $L$ (and $\rho(\{0\})=1$ by definition) and $\{\rho(L):L\in\mathcal{L}(\mathcal{H})\}\subseteq\mathbb{Q}_{\geq 1}\cup\{\infty\}$ is the {\it set of elasticities} of $\mathcal{H}$. Moreover, $\rho(\mathcal{H})=\sup(\{\rho(L):L\in\mathcal{L}(\mathcal{H})\})$ is the elasticity of $\mathcal{H}$. The monoid $\mathcal{H}$ is called {\it fully elastic} if $\mathsf{L}(x)\not=\emptyset$ for each nonunit $x\in\mathcal{H}$ and for every rational number $q$ with $1\leq q<\rho(\mathcal{H})$ there is some $L\in\mathcal{L}(\mathcal{H})$ with $\rho(L)=q$. Sets of elasticities and, in particular, their suprema have been studied for decades (see the survey \cite{An97a} and \cite{CoTr23,Gr22} for some recent progress). Among other things, we know that transfer Krull monoids (including all integrally closed noetherian domains) are fully elastic (see \cite[Theorem 3.1]{GeZh19}, \cite{GeKh22,Zh19}). Strongly primary monoids (including all numerical monoids and one-dimensional local noetherian domains) are not fully elastic (apart from trivial cases; see \cite{BaCh14}, \cite[Theorem 2.2]{ChHoMo06} and \cite[Theorem 5.5]{GeScZh17}).

\medskip
Since $\mathcal{P}_{{\rm fin}}(\mathbb{N}_0)$ contains a cancellative prime element, it is not too difficult to show that $\mathcal{P}_{{\rm fin}}(\mathbb{N}_0)$ is fully elastic (see \cite[Proposition 5.13]{GeKh22}), but the situation is very different for $\mathcal{P}_{{\rm fin},0}(\mathbb{N}_0)$. We formulate the main result of this paper saying, in particular, that $\mathcal{P}_{{\rm fin},0}(\mathbb{N}_0)$ is fully elastic.

\begin{theorem}\label{Theorem 1.1}
Let $G$ be an additive abelian group, let $M\subseteq G$ be an atomic nontorsion submonoid and let $\mathcal{H}^*\subseteq\mathcal{P}_{{\rm fin}}(\mathbb{N}_0)$ be the submonoid generated by $\big\{\{1\},\{2,n\}:n\in\mathbb{N}_{\geq 3}\big\}$. Then the monoids $\mathcal{H}=\mathcal{P}_{{\rm fin}}(M)$ and $\mathcal{H}=\mathcal{P}_{{\rm fin},\times}(M)$ have the following properties:
\begin{enumerate}
\item $\mathcal{H}^*\subseteq\mathcal{L}(\mathcal{H})$.
\item $\{[k,k+2]:k\in\mathbb{N}_{\geq 2}\}\subseteq\mathcal{L}(\mathcal{H})$.
\item For each $q\in\mathbb{Q}_{\geq 1}$, there is some $L\in\mathcal{L}(\mathcal{H})$ with $q=\frac{\max(L)}{\min(L)}$, whence $\mathcal{H}$ is fully elastic $($and $\rho(\mathcal{H})=\infty)$.
\end{enumerate}
\end{theorem}

\medskip
In Section~\ref{Section 2}, we gather the necessary notions and definitions regarding (power) monoids and factorization theory. Clearly, there are nonempty finite subsets $A,B,C\subseteq\mathbb{N}_0$ with $A+B=A+C$, but $B$ and $C$ are distinct (in technical terms, this means that $A$ is not cancellative). The crucial new idea of our approach is the introduction of a weaker concept, called {\it relative cancellativity}.

\section{Notation and terminology}\label{Section 2}

We continue with some background on monoids and factorizations. For the sake of completeness, we introduce (or reiterate) all important concepts in the more general setting of (commutative) monoids.

\medskip
Let $\mathbb{N}$, $\mathbb{N}_0$, $\mathbb{Z}$ and $\mathbb{Q}$ denote the sets of positive integers, nonnegative integers, integers and rational numbers, respectively. Let $x,y\in\mathbb{Z}$ and let $A,B\subseteq\mathbb{Z}$. Then let

\[
A+B=\{a+b:a\in A,b\in B\}
\]

\medskip
be the {\it sumset} of $A$ and $B$. We set

\[
x+A=\{x+a:a\in A\}\quad\textnormal{and}\quad [x,y]=\{z\in\mathbb{Z}:x\leq z\leq y\}.
\]

\medskip
For each $t\in\mathbb{N}$, we set $\mathbb{N}_{\geq t}=\{z\in\mathbb{N}:z\geq t\}$ and for each $r\in\mathbb{Q}$, we set $\mathbb{Q}_{\geq r}=\{z\in\mathbb{Q}:z\geq r\}$. Throughout this note, all monoids are commutative semigroups with identity.

\medskip
Let $M$ be an additively written monoid (with identity $0$) and let $a\in M$. If $b\in M$, then we write $b\mid_M a$ (``$b$ is a divisor of $a$'') if $a=b+c$ for some $c\in M$. Let

\[
M^{\times}=\{a\in M:a+b=0\textnormal{ for some }b\in M\}
\]

\medskip
be the {\it set of units} of $M$ and let

\[
\mathcal{A}(M)=\{x\in M\setminus M^{\times}:\textnormal{ for all }y,z\in M\textnormal{ with }x=y+z\textnormal{, }y\in M^{\times}\textnormal{ or }z\in M^{\times}\}
\]

\medskip
be the {\it set of atoms} of $M$. The element $a$ is called {\it cancellative} if for all $b,c\in M$ with $a+b=a+c$, it follows that $b=c$. We say that $a$ is {\it relatively cancellative} if for all $b,c,d\in M$ with $a=b+c=b+d$, we have $c=d$. The monoid $M$ is called {\it reduced} if $M^{\times}=\{0\}$ and $M$ is called {\it cancellative} if every element of $M$ is cancellative. If $b,c\in M$, then $b$ and $c$ are called {\it relatively prime} if for all $t\in M$ with $t\mid_M b$ and $t\mid_M c$, we have $t\in M^{\times}$.

\medskip
Clearly, every unit of $M$ is cancellative. Also note that every cancellative element of $M$ is relatively cancellative and if $M$ is reduced, then every atom of $M$ is relatively cancellative. Clearly, divisors of cancellative elements are cancellative and divisors of relatively cancellative elements are relatively cancellative. Also note that sums of cancellative elements are cancellative, but sums of relatively cancellative elements can fail to be relatively cancellative (see Example~\ref{Example 3.3} below).

\medskip
We define the {\it length set} of $a$ to be

\[
\mathsf{L}(a)=\left\{n\in\mathbb{N}_0:\textnormal{ there are some atoms }(u_i)_{i=1}^n\textnormal{ of }M\textnormal{ such that }a=\sum_{i=1}^n u_i\right\},
\]

\medskip
Observe that if $x\in M^{\times}\setminus\{0\}$, then $\mathsf{L}(x)=\emptyset$. We say that $M$ is {\it atomic} if every nonunit of $M$ is a finite sum of atoms. Moreover, $M$ is said to be {\it nontorsion} if there exists some $x\in M$ such that $\{nx:n\in\mathbb{N}\}$ is infinite. Define the {\it system of length sets} of $M$ to be $\mathcal{L}(M)=\{\mathsf{L}(x):x\in M,\mathsf{L}(x)\not=\emptyset\}$. (We require that the elements of the system of length sets of $M$ are nonempty, since some elements of $M$ may have an empty length set, as pointed out before.) For each nonempty $L\subseteq\mathbb{N}$, set $\rho(L)=\frac{\sup(L)}{\min(L)}$ and $\rho(\emptyset)=\rho(\{0\})=1$. For each $x\in M$, let $\rho(x)=\rho(\mathsf{L}(x))$, called the {\it elasticity} of $x$. Set $\rho(M)=\sup(\{\rho(x):x\in M\})$, called the {\it elasticity} of $M$. We say that $M$ is {\it fully elastic} if $M$ is atomic and for each $r\in\mathbb{Q}_{\geq 1}$ with $r<\rho(M)$, there is some $x\in M$ such that $\rho(x)=r$. (Note that it is possible that $\rho(M)$ is a rational number and $\rho(x)<\rho(M)$ for each $x\in M$. For this reason, we only consider rational numbers smaller than the elasticity in the definition of fully elastic monoids. This is in line with recent definitions of fully elastic monoids \cite{GeZh19}.)

\medskip
In this paragraph let $M$ be reduced. Define the {\it factorization monoid} of $M$ to be the free abelian monoid $\mathsf{Z}(M)$ with basis $\mathcal{A}(M)$. We always assume that $\mathsf{Z}(M)$ is a multiplicatively written monoid. Observe that $\mathsf{Z}(M)$ is cancellative. Let $|\cdot|:\mathsf{Z}(M)\rightarrow\mathbb{N}_0$ be defined by $|z|=n$ if $n\in\mathbb{N}_0$ and $(u_i)_{i=1}^n$ are atoms of $M$ with $z=\prod_{i=1}^n u_i$. (This map is well-defined, since the representation of $z$ as a product of atoms of $M$ is unique up to order.) Define the {\it set of factorizations} of $x$ by

\[
\mathsf{Z}(x)=\left\{\prod_{i=1}^n u_i:n\in\mathbb{N}_0\textnormal{ and }(u_i)_{i=1}^n\textnormal{ atoms of }M\textnormal{ with }x=\sum_{i=1}^n u_i\right\}
\]

\medskip
for each $x\in M$. It is easy to prove that $\mathsf{L}(x)=\{|z|:z\in\mathsf{Z}(x)\}$ for each $x\in M$. Observe that each two elements $x,y\in\mathsf{Z}(M)$ have a unique greatest common divisor denoted by ${\rm gcd}(x,y)$.

\medskip
Let $\mathcal{P}_{\rm fin}(M)$ be the set of nonempty finite subsets of $M$ and let $+:\mathcal{P}_{\rm fin}(M)\times\mathcal{P}_{\rm fin}(M)\rightarrow\mathcal{P}_{\rm fin}(M)$ be defined by $A+B=\{a+b:a\in A,b\in B\}$ for all $A,B\in\mathcal{P}_{\rm fin}(M)$. Then $\mathcal{P}_{\rm fin}(M)$ (equipped with $+$) is a monoid called the {\it finitary power monoid} of $M$. Let $\mathcal{P}_{{\rm fin},\times}(M)$ be the set of finite subsets $A$ of $M$ such that $A\cap M^{\times}\not=\emptyset$. Then $\mathcal{P}_{{\rm fin},\times}(M)$ is a submonoid of $\mathcal{P}_{\rm fin}(M)$ called the {\it restricted finitary power monoid} of $M$. Let $\mathcal{P}_{{\rm fin},0}(M)$ be the set of finite subsets $A$ of $M$ such that $0\in A$. Then $\mathcal{P}_{{\rm fin},0}(M)$ is a submonoid of $\mathcal{P}_{{\rm fin},\times}(M)$, called the {\it reduced finitary power monoid} of $M$.

\medskip
In Section~\ref{Section 3}, we will use without explicit reference the fact that $\mathcal{P}_{\rm fin}(\mathbb{N}_0)$ and $\mathcal{P}_{{\rm fin},\times}(\mathbb{N}_0)=\mathcal{P}_{{\rm fin},0}(\mathbb{N}_0)$ are reduced atomic monoids and

\[
\mathcal{L}(\mathcal{P}_{{\rm fin},0}(\mathbb{N}_0))\subseteq\mathcal{L}(\mathcal{P}_{\rm fin}(\mathbb{N}_0))\subseteq\{\{0\},\{1\}\}\cup\{A\subseteq\mathbb{N}_{\geq 2}:A\textnormal{ is finite and nonempty}\}.
\]

\medskip
For more details, see \cite[Proposition 3.2 and Corollary 3.6]{FaTr18}. We will also use the fact that $A+(B\cup C)=(A+B)\cup (A+C)$ for all $A,B,C\subseteq\mathbb{Z}$ without further mention.

\section{Results}\label{Section 3}

We start with a result that will enable us to show that sumsets of length sets of relatively cancellative elements are again length sets of relatively cancellative elements. Note that for each $v\in\mathbb{Z}$, $|v|$ denotes the absolute value of $v$ in the proof of Proposition~\ref{Proposition 3.1} below (and it is not to be confused with $|\cdot|$ defined on the factorization monoid). For each $X\subseteq\mathbb{Z}$, let ${\rm gcd}(X)\in\mathbb{N}_0$ be the greatest common divisor of $X$ (where $\mathbb{Z}$ is viewed as a monoid equipped with the ordinary multiplication). Observe that ${\rm gcd}(\emptyset)={\rm gcd}(\{0\})=0$ and ${\rm gcd}(X)>0$ if $X\nsubseteq\{0\}$.

\begin{proposition}\label{Proposition 3.1}
Let $\mathcal{H}=\mathcal{P}_{{\rm fin},0}(\mathbb{N}_0)$ and let $X,Y\in\mathcal{H}$ be relatively cancellative such that ${\rm gcd}(Y)>2\max(X)$.
\begin{enumerate}
\item[(1)] For all $x,x^{\prime},x^{\prime\prime}\in X$ and $y,y^{\prime},y^{\prime\prime}\in Y$ with $x+y=x^{\prime}+x^{\prime\prime}+y^{\prime}+y^{\prime\prime}$, it follows that $x=x^{\prime}+x^{\prime\prime}$ and $y=y^{\prime}+y^{\prime\prime}$.
\item[(2)] For all $A,B\in\mathcal{H}$ with $A+B=X+Y$, we have $X=(A\cap X)+(B\cap X)$, $Y=(A\cap Y)+(B\cap Y)$ and $A=(A\cap X)+(A\cap Y)$.
\item[(3)] $X+Y$ is relatively cancellative, $\mathsf{Z}(X+Y)=\mathsf{Z}(X)\mathsf{Z}(Y)$ and $\mathsf{L}(X+Y)=\mathsf{L}(X)+\mathsf{L}(Y)$.
\end{enumerate}
\end{proposition}

\begin{proof}
(1) Let $x,x^{\prime},x^{\prime\prime}\in X$ and $y,y^{\prime},y^{\prime\prime}\in Y$ with $x+y=x^{\prime}+x^{\prime\prime}+y^{\prime}+y^{\prime\prime}$. Then

\[
{\rm gcd}(Y)\mid y-y^{\prime}-y^{\prime\prime}=x^{\prime}+x^{\prime\prime}-x,
\]

\medskip
and hence $|x^{\prime}+x^{\prime\prime}-x|=a{\rm gcd}(Y)$ for some $a\in\mathbb{N}_0$. Since

\[
|x^{\prime}+x^{\prime\prime}-x|\leq\max(\{x,x^{\prime}+x^{\prime\prime}\})\leq 2\max(X)<{\rm gcd}(Y),
\]

\medskip
we infer that $a=0$. Consequently, $x=x^{\prime}+x^{\prime\prime}$ and $y=y^{\prime}+y^{\prime\prime}$.

\bigskip
(2) Let $A,B\in\mathcal{H}$ be such that $A+B=X+Y$. Set

\begin{align*}
&A^{\prime}=\{x\in X:x+y\in A\textnormal{ for some }y\in Y\}\textnormal{, }A^{\prime\prime}=\{y\in Y:x+y\in A\textnormal{ for some }x\in X\},\\
&B^{\prime}=\{x\in X:x+y\in B\textnormal{ for some }y\in Y\}\textnormal{, }B^{\prime\prime}=\{y\in Y:x+y\in B\textnormal{ for some }x\in X\}.
\end{align*}

\medskip
Then $A\cap X\subseteq A^{\prime}$, $A\cap Y\subseteq A^{\prime\prime}$, $B\cap X\subseteq B^{\prime}$ and $B\cap Y\subseteq B^{\prime\prime}$. In particular, $A^{\prime},A^{\prime\prime},B^{\prime},B^{\prime\prime}\in\mathcal{H}$.

\medskip
First we show that

\[
X=(A\cap X)+(B\cap X)=(A\cap X)+B^{\prime}=A^{\prime}+B^{\prime}.
\]

\medskip
Let $z\in X$. Then $z=a+b$ for some $a\in A$ and $b\in B$. Clearly, there are some $x^{\prime},x^{\prime\prime}\in X$ and $y^{\prime},y^{\prime\prime}\in Y$ such that $a=x^{\prime}+y^{\prime}$ and $b=x^{\prime\prime}+y^{\prime\prime}$. Since $z=x^{\prime}+x^{\prime\prime}+y^{\prime}+y^{\prime\prime}$, it follows by (1) that $z=x^{\prime}+x^{\prime\prime}$ and $y^{\prime}=y^{\prime\prime}=0$. Observe that $a\in A\cap X$ and $b\in B\cap X$, and hence $z\in (A\cap X)+(B\cap X)$. This shows that

\[
X\subseteq (A\cap X)+(B\cap X)\subseteq (A\cap X)+B^{\prime}\subseteq A^{\prime}+B^{\prime}.
\]

\medskip
Let $c\in A^{\prime}$ and $d\in B^{\prime}$. Then there are some $u,v\in Y$ such that $c+u\in A$ and $d+v\in B$. This implies that $c+d+u+v\in A+B=X+Y$, and thus $c+d\in X$ by (1).

\medskip
It can be shown along the same lines that

\[
Y=(A\cap Y)+(B\cap Y)=(A\cap Y)+B^{\prime\prime}=A^{\prime\prime}+B^{\prime\prime}.
\]

\medskip
Since $X$ and $Y$ are relatively cancellative, it follows that $A^{\prime}=A\cap X$ and $A^{\prime\prime}=A\cap Y$.

\medskip
It remains to prove that $A=A^{\prime}+A^{\prime\prime}$. If $w\in A$, then $w=x+y$ for some $x\in X$ and $y\in Y$, and thus $x\in A^{\prime}$, $y\in A^{\prime\prime}$ and $w=x+y\in A^{\prime}+A^{\prime\prime}$. Consequently, $A\subseteq A^{\prime}+A^{\prime\prime}$. It follows along the same lines that $B\subseteq B^{\prime}+B^{\prime\prime}$. Now we show that $A^{\prime}+A^{\prime\prime}\subseteq A$. Let $a^{\prime}\in A^{\prime}$ and $a^{\prime\prime}\in A^{\prime\prime}$. The next step is to prove that there is some $v^{\prime}\in B^{\prime}$ such that for all $e\in A^{\prime}$ and $f\in B^{\prime}$ with $a^{\prime}+v^{\prime}=e+f$, we have $a^{\prime}=e$.

\medskip
If $a^{\prime}=0$, then set $v^{\prime}=0$. Now let $a^{\prime}\not=0$. Note that $A^{\prime}\setminus\{a^{\prime}\}\in\mathcal{H}$ and $(A^{\prime}\setminus\{a^{\prime}\})+B^{\prime}\subsetneq X$ (since $X$ is relatively cancellative and $X=A^{\prime}+B^{\prime}$). There is some $x\in X\setminus ((A^{\prime}\setminus\{a^{\prime}\})+B^{\prime})$. It is clear that $x=a^{\prime}+v^{\prime}$ for some $v^{\prime}\in B^{\prime}$. Let $e\in A^{\prime}$ and $f\in B^{\prime}$ be such that $a^{\prime}+v^{\prime}=e+f$. Then $e+f=x\in X\setminus ((A^{\prime}\setminus\{a^{\prime}\})+B^{\prime})$, and so $a^{\prime}=e$.

\medskip
It can be shown along similar lines that there is some $v^{\prime\prime}\in B^{\prime\prime}$ such that for all $e\in A^{\prime\prime}$ and $f\in B^{\prime\prime}$ with $a^{\prime\prime}+v^{\prime\prime}=e+f$, we have $a^{\prime\prime}=e$. Observe that

\[
a^{\prime}+v^{\prime}+a^{\prime\prime}+v^{\prime\prime}\in A^{\prime}+B^{\prime}+A^{\prime\prime}+B^{\prime\prime}=X+Y=A+B,
\]

\medskip
and thus $a^{\prime}+v^{\prime}+a^{\prime\prime}+v^{\prime\prime}=a+b$ for some $a\in A$ and $b\in B$. Since $A\subseteq A^{\prime}+A^{\prime\prime}$ and $B\subseteq B^{\prime}+B^{\prime\prime}$, there are some $x^{\prime}\in A^{\prime}$, $x^{\prime\prime}\in A^{\prime\prime}$, $y^{\prime}\in B^{\prime}$ and $y^{\prime\prime}\in B^{\prime\prime}$ such that $a=x^{\prime}+x^{\prime\prime}$ and $b=y^{\prime}+y^{\prime\prime}$. Since

\begin{align*}
&a^{\prime}+v^{\prime}\in A^{\prime}+B^{\prime}=X\textnormal{, }a^{\prime\prime}+v^{\prime\prime}\in A^{\prime\prime}+B^{\prime\prime}=Y\textnormal{, }\\
&x^{\prime},y^{\prime}\in X\textnormal{, }x^{\prime\prime},y^{\prime\prime}\in Y\textnormal{,}\\
&a^{\prime}+v^{\prime}+a^{\prime\prime}+v^{\prime\prime}=x^{\prime}+y^{\prime}+x^{\prime\prime}+y^{\prime\prime},
\end{align*}

\medskip
we infer by (1) that $a^{\prime}+v^{\prime}=x^{\prime}+y^{\prime}$ and $a^{\prime\prime}+v^{\prime\prime}=x^{\prime\prime}+y^{\prime\prime}$. Therefore, $a^{\prime}=x^{\prime}$ and $a^{\prime\prime}=x^{\prime\prime}$. This implies that $a^{\prime}+a^{\prime\prime}=x^{\prime}+x^{\prime\prime}=a\in A$.

\bigskip
(3) First we show that $X+Y$ is relatively cancellative. Let $A,B,C\in\mathcal{H}$ be such that

\[
X+Y=A+B=A+C.
\]

\medskip
We infer by (2) that

\begin{align*}
X&=(A\cap X)+(B\cap X)=(A\cap X)+(C\cap X),\\
Y&=(A\cap Y)+(B\cap Y)=(A\cap Y)+(C\cap Y),\\
B&=(B\cap X)+(B\cap Y)\textnormal{, }C=(C\cap X)+(C\cap Y).
\end{align*}

\medskip
Since $X$ and $Y$ are relatively cancellative, it follows that $B\cap X=C\cap X$ and $B\cap Y=C\cap Y$, and thus

\[
B=(B\cap X)+(B\cap Y)=(C\cap X)+(C\cap Y)=C.
\]

\medskip
Next we prove that $\mathsf{Z}(X+Y)=\mathsf{Z}(X)\mathsf{Z}(Y)$. Clearly, $\mathsf{Z}(X)\mathsf{Z}(Y)\subseteq\mathsf{Z}(X+Y)$. Now let $z\in\mathsf{Z}(X+Y)$. There are some $n\in\mathbb{N}$ and atoms $(A_i)_{i=1}^n$ of $\mathcal{H}$ such that $z=\prod_{i=1}^n A_i$. If $j\in [1,n]$, then $A_j\mid_{\mathcal{H}}\sum_{i=1}^n A_i=X+Y$, and hence $A_j=(A_j\cap X)+(A_j\cap Y)$ by (2), which implies that $A_j\subseteq X$ or $A_j\subseteq Y$ (since $A_j$ is an atom of $\mathcal{H}$). Set

\[
z^{\prime}=\prod_{i=1,A_i\subseteq X}^n A_i\textnormal{, }z^{\prime\prime}=\prod_{i=1,A_i\subseteq Y}^n A_i\textnormal{, }\overline{A}=\sum_{i=1,A_i\subseteq X}^n A_i\textnormal{, }\overline{B}=\sum_{i=1,A_i\subseteq Y}^n A_i.
\]

\medskip
Since $A_j\nsubseteq\{0\}=X\cap Y$ for each $j\in [1,n]$, we infer that $z=z^{\prime}z^{\prime\prime}$, and thus $X+Y=\overline{A}+\overline{B}$. Note that

\[
(X+X)\cap (X+Y)\subseteq X\textnormal{ and }(Y+Y)\cap (X+Y)\subseteq Y
\]

\medskip
by (1). Using this fact, we find by induction that $\overline{A}\subseteq X$ and $\overline{B}\subseteq Y$. In particular, we deduce that $X+Y=\overline{A}+Y=X+\overline{B}$. Since $X+Y$ is relatively cancellative, this implies that $X=\overline{A}$ and $Y=\overline{B}$, and thus $z^{\prime}\in\mathsf{Z}(X)$ and $z^{\prime\prime}\in\mathsf{Z}(Y)$. Therefore, $z=z^{\prime}z^{\prime\prime}\in\mathsf{Z}(X)\mathsf{Z}(Y)$.

\medskip
Since $\mathsf{Z}(X+Y)=\mathsf{Z}(X)\mathsf{Z}(Y)$, clearly $\mathsf{L}(X+Y)=\mathsf{L}(X)+\mathsf{L}(Y)$.
\end{proof}

Now we prove the desired additivity of length sets of relatively cancellative elements.

\begin{corollary}\label{Corollary 3.2}
Let $\mathcal{H}=\mathcal{P}_{{\rm fin},0}(\mathbb{N}_0)$ and let $X,Y\in\mathcal{H}$ be relatively cancellative. Then $\mathsf{L}(X)+\mathsf{L}(Y)=\mathsf{L}(W)$ for some relatively cancellative $W\in\mathcal{H}$.
\end{corollary}

\begin{proof}
Without restriction, we can assume that $Y\not=\{0\}$. Let $d\in\mathbb{N}$ be such that $d>2\max(X)$ and set $Z=\{da:a\in Y\}$. Then $Z\in\mathcal{H}$ is relatively cancellative, $\mathsf{L}(Z)$ equals $\mathsf{L}(Y)$ and we have ${\rm gcd}(Z)\geq d>2\max(X)$. Set $W=X+Z$. By Proposition~\ref{Proposition 3.1}(3), $W\in\mathcal{H}$ is relatively cancellative and

\[
\mathsf{L}(X)+\mathsf{L}(Y)=\mathsf{L}(X)+\mathsf{L}(Z)=\mathsf{L}(W).\qed
\]
\end{proof}

Next we want to point out that all the assumptions in Proposition~\ref{Proposition 3.1} are necessary.

\begin{example}\label{Example 3.3}
Let $\mathcal{H}=\mathcal{P}_{{\rm fin},0}(\mathbb{N}_0)$, $A=\{0,1,2,3\}$, $B=\{0,7\}$, $C=\{0,1\}$, $D=\{0,3,6,9\}$, $E=\{0,2\}$.
\begin{enumerate}
\item[(1)] $A$ and $D$ are not relatively cancellative and $B$, $C$ and $E$ are relatively cancellative.
\item[(2)] ${\rm gcd}(B)>2\max(A)$, ${\rm gcd}(D)>2\max(C)$ and ${\rm gcd}(E)=2\max(C)$.
\item[(3)] $\mathsf{L}(A+B)=\{2,3,4\}\supsetneq\{3,4\}=\mathsf{L}(A)+\mathsf{L}(B)$.
\item[(4)] $\mathsf{L}(C+D)=\{2,3,4\}\supsetneq\{3,4\}=\mathsf{L}(C)+\mathsf{L}(D)$.
\item[(5)] $\mathsf{L}(C+E)=\{2,3\}\supsetneq\{2\}=\mathsf{L}(C)+\mathsf{L}(E)$.
\end{enumerate}
\end{example}

\begin{proof}
Clearly, $B$, $C$ and $E$ are atoms of $\mathcal{H}$, and

\begin{align*}
&A=C+\{0,1,2\}=C+E=C+C+C\textnormal{,}\\
&D=\{0,3\}+\{0,3,6\}=\{0,3\}+\{0,6\}=\{0,3\}+\{0,3\}+\{0,3\}.
\end{align*}

\medskip
In particular, $A$ and $D$ are not relatively cancellative. Moreover, $A+B=C+\{0,1,2,7,9\}$ and $\{0,1,2,7,9\}$ is an atom of $\mathcal{H}$. Also note that $C+D=\{0,3\}+\{0,1,4,6,7\}$ and $\{0,1,4,6,7\}$ is an atom of $\mathcal{H}$. The rest is straightforward.
\end{proof}

We now provide a simple sufficient criterion to obtain relatively cancellative elements. Observe that not all relatively cancellative elements satisfy this criterion. (For instance, we can use Proposition~\ref{Proposition 3.1} to construct a multitude of counterexamples.)

\begin{lemma}\label{Lemma 3.4}
Let $\mathcal{H}=\mathcal{P}_{{\rm fin},0}(\mathbb{N}_0)$ and let $X\in\mathcal{H}$ be such that for all $u^{\prime},v^{\prime}\in\mathsf{Z}(X)$ with ${\gcd}(u^{\prime},v^{\prime})\not=1$, we have $u^{\prime}=v^{\prime}$. Then $X$ is relatively cancellative.
\end{lemma}

\begin{proof}
Let $A,B,C\in\mathcal{H}$ be such that $A+B=A+C=X$. Since $\mathcal{H}$ is reduced and atomic, it is clear that $\mathsf{Z}(A)$, $\mathsf{Z}(B)$ and $\mathsf{Z}(C)$ are nonempty. Pick $u\in\mathsf{Z}(A)$, $v\in\mathsf{Z}(B)$ and $w\in\mathsf{Z}(C)$. Since $A+B=A+C=X$, it follows that $uv,uw\in\mathsf{Z}(X)$. If ${\rm gcd}(uv,uw)=1$, then $u=1$, and hence $A=\{0\}$ and $B=C$. Now let ${\rm gcd}(uv,uw)\not=1$. Then $uv=uw$, and thus $v=w$. Therefore, $B=C$.
\end{proof}

The next step is to show that each set of the form $\{2,n\}$ with $n\in\mathbb{N}_{\geq 3}$ is indeed the length set of some relatively cancellative element of $\mathcal{P}_{{\rm fin},0}(\mathbb{N}_0)$. (Note that the length sets constructed in \cite[Proposition 4.10]{FaTr18} are of the form $\{2,n\}$, but they do not stem from relatively cancellative elements.) For each set $A$, we let $|A|$ denote the cardinality of $A$ in the proof of Theorem~\ref{Theorem 3.5} below. (It is not to be confused with $|\cdot|$ defined on the factorization monoid.)

\begin{theorem}\label{Theorem 3.5}
Let $\mathcal{H}=\mathcal{P}_{{\rm fin},0}(\mathbb{N}_0)$ and let $(n_j)_{j=0}^{\infty}$, $(A_j)_{j=0}^{\infty}$, $(B_j)_{j=0}^{\infty}$, $(C_j)_{j=0}^{\infty}$, $(D_j)_{j=0}^{\infty}$ and $(S_j)_{j=0}^{\infty}$ be defined recursively as follows: Set $n_0=0$, $A_0=\{0,1\}$, $B_0=\{0\}$, $C_0=\{0\}$, $D_0=\{0,1\}$ and $S_0=\{0,1\}$; for all $i\in\mathbb{N}_0$ set $n_{i+1}\in\mathbb{N}$ with $n_{i+1}\geq 3\max(D_i)$, $A_{i+1}=A_i\cup\{1+n_{i+1}\}$, $B_{i+1}=B_i\cup (n_{i+1}+D_i)$, $C_{i+1}=C_i\cup\{n_{i+1}\}$, $D_{i+1}=D_i+\{0,1+n_{i+1}\}$ and $S_{i+1}=C_{i+1}+D_{i+1}$.
\begin{enumerate}
\item[(1)] For each $i\in\mathbb{N}$, $A_i,B_i,C_i\in\mathcal{A}(\mathcal{H})$, $D_i\in\mathcal{H}$, $\mathsf{L}(D_i)=\{i+1\}$, $\mathsf{Z}(D_i)=\{\prod_{j=0}^i\{0,1+n_j\}\}$, $\max(A_i\cup B_i\cup C_i)<\max(D_i)$, $A_i=\{0\}\cup (1+C_i)$ and $A_i+B_i=C_i+D_i$.
\item[(2)] For each $i\in\mathbb{N}$, $\mathsf{Z}(S_i)=\{A_iB_i,C_i\prod_{j=0}^i\{0,1+n_j\}\}$, $\mathsf{L}(S_i)=\{2,i+2\}$ and for all $u,v\in\mathsf{Z}(S_i)$ with ${\rm gcd}(u,v)\not=1$, we have $u=v$.
\end{enumerate}
\end{theorem}

\begin{proof}
(1) We prove the statement by induction on $i$. Observe that

\begin{align*}
&A_1,B_1,C_1\in\mathcal{A}(\mathcal{H})\textnormal{, }D_1\in\mathcal{H}\textnormal{, }\mathsf{Z}(D_1)=\{\{0,1\}\{0,1+n_1\}\}\textnormal{, }\mathsf{L}(D_1)=\{2\}\textnormal{,}\\
&\max(A_1\cup B_1\cup C_1)=1+n_1<2+n_1=\max(D_1)\textnormal{, }A_1=\{0\}\cup (1+C_1)\textnormal{,}\\
&A_1+B_1=\{0,1,n_1,1+n_1,2+n_1,1+2n_1,2+2n_1\}=C_1+D_1.
\end{align*}

\medskip
Now let $i\in\mathbb{N}$ and assume that the statement is true for $i$. It is clear that $A_{i+1},B_{i+1},C_{i+1},D_{i+1}\in\mathcal{H}$. Since $\mathsf{Z}(D_i)=\{\prod_{j=0}^i\{0,1+n_j\}\}$ and $1+n_{i+1}>2\max(D_i)$, it follows from Proposition~\ref{Proposition 3.1} and Lemma~\ref{Lemma 3.4} that $\mathsf{Z}(D_{i+1})=\{\prod_{j=0}^{i+1}\{0,1+n_j\}\}$ and $\mathsf{L}(D_{i+1})=\{i+2\}$. Since $\max(A_i\cup B_i\cup C_i)<\max(D_i)<n_{i+1}$, we have $\max(A_{i+1})=1+n_{i+1}$, $\max(B_{i+1})=n_{i+1}+\max(D_i)$, $\max(C_{i+1})=n_{i+1}$ and $\max(D_{i+1})=\max(D_i)+1+n_{i+1}$, and hence $\max(A_{i+1}\cup B_{i+1}\cup C_{i+1})<\max(D_{i+1})$. Moreover,

\begin{align*}
&2\max(A_{i+1}\setminus\{\max(A_{i+1})\})=2\max(A_i)<1+n_{i+1}=\max(A_{i+1})\textnormal{,}\\
&2\max(C_{i+1}\setminus\{\max(C_{i+1})\})=2\max(C_i)<n_{i+1}=\max(C_{i+1}),
\end{align*}

\medskip
and thus $A_{i+1},C_{i+1}\in\mathcal{A}(\mathcal{H})$. Since $A_i=\{0\}\cup (1+C_i)$, it follows that

\[
A_{i+1}=\{0,1+n_{i+1}\}\cup (1+C_i)=\{0\}\cup (1+C_{i+1}).
\]

\medskip
Also note that $B_i\subseteq A_i+B_i=C_i+D_i$. We infer that

\begin{align*}
A_{i+1}+B_{i+1}
&=(A_i+B_i)\cup (1+n_{i+1}+B_i)\cup (n_{i+1}+A_i+D_i)\cup (1+2n_{i+1}+D_i)\\
&=(C_i+D_i)\cup (1+n_{i+1}+B_i)\cup (n_{i+1}+D_i)\cup (1+n_{i+1}+C_i+D_i)\cup (1+2n_{i+1}+D_i)\\
&=(C_i+D_i)\cup (n_{i+1}+D_i)\cup (1+n_{i+1}+C_i+D_i)\cup (1+2n_{i+1}+D_i)\\
&=(C_i\cup\{n_{i+1}\})+(D_i\cup (1+n_{i+1}+D_i))=C_{i+1}+D_{i+1}.
\end{align*}

\medskip
It remains to prove that $B_{i+1}\in\mathcal{A}(\mathcal{H})$. Let $X,Y\in\mathcal{H}$ be such that $B_{i+1}=X+Y$ and $\max(X)\leq\max(Y)$. We need to show that $X=\{0\}$. Since

\begin{align*}
&2n_{i+1}>n_{i+1}+\max(D_i)=\max(X)+\max(Y)\geq 2\max(X)\textnormal{,}\\
&B_{i+1}\subseteq [0,\max(D_i)]\cup [n_{i+1},n_{i+1}+\max(D_i)],
\end{align*}

\medskip
we have $\max(X)\leq\max(D_i)<n_{i+1}$. This implies that $X\subseteq B_i$ (since $X\subseteq B_{i+1}=B_i\cup (n_{i+1}+D_i)$). Note that $n_{i+1}\in B_{i+1}=X+Y$, and hence there are some $x\in X$ and $y\in Y$ such that $n_{i+1}=x+y$. If $y\leq\max(B_i)$, then $n_{i+1}\leq 2\max(D_i)$, a contradiction. Therefore, $y>\max(B_i)$, and thus $y\in B_{i+1}\setminus B_i$. Consequently, there is some $z\in D_i$ such that $y=n_{i+1}+z$. We infer that $x=0$ and $n_{i+1}\in Y$. Set $Z=\{a\in D_i:n_{i+1}+a\in Y\}$. Then $Z\in\mathcal{H}$. It follows that

\[
Y=Y\cap B_{i+1}=(Y\cap B_i)\cup (Y\cap (n_{i+1}+D_i))=(Y\cap B_i)\cup (n_{i+1}+Z),
\]

\medskip
and hence

\[
B_i\cup (n_{i+1}+D_i)=X+Y=(X+(Y\cap B_i))\cup (n_{i+1}+X+Z).
\]

\medskip
Since $\max(B_i)<n_{i+1}$ and $\max(X+(Y\cap B_i))<n_{i+1}$, this implies that $B_i=X+(Y\cap B_i)$ and $n_{i+1}+D_i=n_{i+1}+X+Z$. If $Y\cap B_i=\{0\}$, then $n_1\in B_i\subseteq B_i+Z=X+Z=D_i$, a contradiction. Therefore, $Y\cap B_i\not=\{0\}$. Since $B_i\in\mathcal{A}(\mathcal{H})$, we have $X=\{0\}$.

\bigskip
(2) We prove by induction on $i$ that for each $i\in\mathbb{N}$,

\[
\mathsf{Z}(S_i)=\{A_iB_i,C_i\prod_{j=0}^i\{0,1+n_j\}\}.
\]

First we show that $\mathsf{Z}(S_1)=\{A_1B_1,C_1\{0,1\}\{0,1+n_1\}\}$. By (1) we have $\{A_1B_1,C_1\{0,1\}\{0,1+n_1\}\}\subseteq\mathsf{Z}(S_1)$. Also note that

\[
S_1=\{0,1,n_1,1+n_1,2+n_1,1+2n_1,2+2n_1\}.
\]

\medskip
Let $z\in\mathsf{Z}(S_1)$. There are some $t\in\mathbb{N}_{\geq 2}$ and atoms $(X_i)_{i=1}^t$ of $\mathcal{H}$ such that $z=\prod_{i=1}^t X_i$ and $\max(X_j)\leq\max(X_{j+1})$ for each $j\in [1,t-1]$. First let $t\geq 3$. Since $2,2n_1\not\in S_1$ and $\sum_{i=1}^t\max(X_i)=2+2n_1$, we infer that $t=3$, $\max(X_1)=1$, $\max(X_2)=n_1$ and $\max(X_3)=1+n_1$. This implies that

\[
z=\{0,1\}\{0,n_1\}\{0,1+n_1\}=C_1\{0,1\}\{0,1+n_1\}.
\]

\medskip
Let $t=2$. Then

\[
(\max(X_1),\max(X_2))\in\{(1,1+2n_1),(n_1,2+n_1),(1+n_1,1+n_1)\}.
\]

\medskip
If $(\max(X_1),\max(X_2))=(1,1+2n_1)$, then $X_1=\{0,1\}$ and

\[
X_2=\{0,n_1,1+n_1,1+2n_1\}=\{0,n_1\}+\{0,1+n_1\}\not\in\mathcal{A}(\mathcal{H}),
\]

\medskip
a contradiction. If $(\max(X_1),\max(X_2))=(n_1,2+n_1)$, then $X_1=\{0,n_1\}$ (since $3+n_1\not\in S_1$) and

\[
X_2=\{0,1,1+n_1,2+n_1\}=\{0,1\}+\{0,1+n_1\}\not\in\mathcal{A}(\mathcal{H}),
\]

\medskip
a contradiction. Therefore, $\max(X_1)=\max(X_2)=1+n_1$. It is obvious that

\[
|\{u\in\{1,2\}:1\in X_u\}|=|\{u\in\{1,2\}:n_1\in X_u\}|=1.
\]

\medskip
If there is some $u\in\{1,2\}$ with $\{1,n_1\}\subseteq X_u$, then $X_u=\{0,1,n_1,1+n_1\}=\{0,1\}+\{0,n_1\}\not\in\mathcal{A}(\mathcal{H})$, a contradiction. We infer that $z=X_1X_2=\{0,1,1+n_1\}\{0,n_1,1+n_1\}=A_1B_1$.

\medskip
Let $i\in\mathbb{N}$ be such that $\mathsf{Z}(S_i)=\{A_iB_i,C_i\prod_{j=0}^i\{0,1+n_j\}\}$. Then

\[
\left\{A_{i+1}B_{i+1},C_{i+1}\prod_{j=0}^{i+1}\{0,1+n_j\}\right\}\subseteq\mathsf{Z}(S_{i+1})
\]

\medskip
by (1). Let $z\in\mathsf{Z}(S_{i+1})$. Then there are some $t\in\mathbb{N}$ and atoms $(Y_j)_{j=1}^{t+1}$ of $\mathcal{H}$ such that $z=\prod_{j=1}^{t+1} Y_j$. For each $j\in [0,i+1]$ set $m_j=1+n_j$ and for each $E\subseteq [0,i+1]$ set $\sum_E=\sum_{k\in E} m_k$. Then

\begin{align*}
&A_{i+1}=\{0\}\cup\{m_j:j\in [0,i+1]\}\textnormal{, }B_{i+1}=\left\{\sum_E-1:\emptyset\not=E\subseteq [0,i+1]\right\}\textnormal{, }\\
&C_{i+1}=\{n_j:j\in [0,i+1]\}\textnormal{, }D_{i+1}=\left\{\sum_E:E\subseteq [0,i+1]\right\}\textnormal{,}\\
&S_{i+1}=\left\{\sum_E-1,m_r+\sum_E-1:E\subseteq [0,i+1],r\in E\right\}.
\end{align*}

\medskip
Clearly, there is precisely one $\ell\in [1,t+1]$ such that $n_1\in Y_{\ell}$ (since $[1,n_1]\cap S_{i+1}=\{1,n_1\}$ and $2n_1\not\in S_{i+1}$). Without restriction, let $n_1\in Y_{t+1}$. Set $X=\sum_{j=1}^t Y_j$ and set $Y=Y_{t+1}$. Then $S_{i+1}=X+Y$.

\medskip
\textsc{Claim} 1: {\it If $(\lambda_j)_{j=0}^{i+1},(\mu_j)_{j=0}^{i+1}\in [0,3]^{[0,i+1]}$ are such that $\sum_{j=0}^{i+1}\lambda_jm_j=\sum_{j=0}^{i+1}\mu_jm_j$, then $\lambda_j=\mu_j$ for each $j\in [0,i+1]$.}

\medskip
{\it Proof}. Let $(\lambda_j)_{j=0}^{i+1},(\mu_j)_{j=0}^{i+1}\in [0,3]^{[0,i+1]}$ be such that $\sum_{j=0}^{i+1}\lambda_jm_j=\sum_{j=0}^{i+1}\mu_jm_j$. Assume that the statement is not true. Then there is a maximal $k\in [1,i+1]$ such that $\lambda_k\not=\mu_k$. Without restriction, let $\lambda_k<\mu_k$. It follows that

\[
m_k\leq\sum_{j=0}^{k-1}\mu_jm_j+(\mu_k-\lambda_k)m_k=\sum_{j=0}^{k-1}\lambda_jm_j\leq 3\sum_{j=0}^{k-1} m_j=3\max(D_{k-1})<m_k,
\]

\medskip
a contradiction. Therefore, $\lambda_j=\mu_j$ for each $j\in [0,i+1]$.\qed

\medskip
\textsc{Claim} 2: {\it For all $r,s\in [1,i+1]$ and $E\subseteq [0,i+1]$ with $n_r+n_s+\sum_E\in S_{i+1}$, it follows that $0\in E$ and $\{r,s\}\nsubseteq E$.}

\medskip
{\it Proof}. Let $r,s\in [1,i+1]$ and $E\subseteq [0,i+1]$ be such that $n_r+n_s+\sum_E\in S_{i+1}$. Then

\[
m_r+m_s+\sum_E=m_0+m_{\ell}+\sum_F
\]

\medskip
for some $\ell\in [0,i+1]$ and $F\subseteq [0,i+1]$. It follows by Claim 1 that $0\in E$ and $\{r,s\}\nsubseteq E$.\qed

\medskip
\textsc{Claim} 3: {\it $C_{i+1}\subseteq Y$.}

\medskip
{\it Proof}. It suffices to prove by induction that $n_j\in Y$ for each $j\in [0,i+1]$. Clearly, $n_0,n_1\in Y$. Now let $j\in [1,i]$ be such that $n_j\in Y$. Since $n_{j+1}\in S_{i+1}=X+Y$, there are some $x\in X$ and $y\in Y$ with $n_{j+1}=x+y$. Consequently, there are some $r,s\in [0,i+1]$ and $E,F\subseteq [0,i+1]$ such that $x=n_r+\sum_E$ and $y=n_s+\sum_F$. If $r\geq j+1$, then $x=n_{j+1}$ and $n_{j+1}+n_j\in X+Y=S_{i+1}$, which contradicts Claim 2. Therefore, $r\leq j$. Clearly, $\max(E)\leq j$ and $\max(F)\leq j$. Assume that $s\leq j$. If $r<j$, then

\[
n_{j+1}=x+y=n_r+n_s+\sum_E+\sum_F<3\sum_{[0,j]}=3\max(D_j)\leq n_{j+1},
\]

\medskip
a contradiction. We infer that $r=j$. Set $E^{\prime}=(E\setminus\{0\})\cup\{r\}$. Since $2n_r+\sum_E=x+n_j\in X+Y=S_{i+1}$, it follows that $0\in E$ and $r\not\in E$ by Claim 2. Observe that $x=n_r+1+\sum_{E\setminus\{0\}}=\sum_{E^{\prime}}$. Moreover,

\[
n_{j+1}=\sum_{E^{\prime}}+n_s+\sum_F<3\sum_{[0,j]}=3\max(D_j)\leq n_{j+1},
\]

\medskip
a contradiction. We have that $s\geq j+1$, and thus $n_{j+1}=y\in Y$.\qed

\medskip
\textsc{Claim} 4: {\it $X\subseteq D_{i+1}$.}

\medskip
{\it Proof}. Let $x\in X$. Then there are some $r\in [0,i+1]$ and $E\subseteq [0,i+1]$ such that $x=n_r+\sum_E$. If $r=0$, then $x=\sum_E\in D_{i+1}$. Now let $r>0$. It follows that $2n_r+\sum_E=x+n_r\in X+Y=S_{i+1}$ by Claim 3, and hence $0\in E$ and $r\not\in E$ by Claim 2. This implies that

\[
x=n_r+1+\sum_{E\setminus\{0\}}=\sum_{(E\setminus\{0\})\cup\{r\}}\in D_{i+1}.\qed
\]

\medskip
The following statements are simple consequences of Claims 1 and 4 and the fact that $S_{i+1}=X+Y$.

\begin{itemize}
\item[($\ast_1$)] For each $\emptyset\not=E\subseteq [0,i+1]$, there are some $F,G\subseteq [0,i+1]$ such that $E=F\cup G$, $F\cap G=\emptyset$, $\sum_F\in X$ and $\sum_G-1\in Y$.
\item[($\ast_2$)] For each $E\subseteq [0,i+1]$ and $r\in E$, there are some $F,G\subseteq [0,i+1]$ such that $E=F\cup G$, $F\cap G\subseteq\{r\}$, $r\in G$, $\sum_F\in X$ and $\sum_{\{r\}\setminus F}+\sum_G-1\in Y$.
\item[($\ast_3$)] If $F,G\subseteq [0,i+1]$ are such that $\sum_F+\sum_G-1\in S_{i+1}$, then $|F\cap G|\leq 1$.
\item[($\ast_4$)] If $F,G\subseteq [0,i+1]$ and $r\in G$ are such that $\sum_F+m_r+\sum_G-1\in S_{i+1}$, then $F\cap G=\emptyset$.
\end{itemize}

\medskip
Set $X^{\prime}=X\cap [0,n_{i+1}-1]$ and $Y^{\prime}=Y\cap [0,n_{i+1}-1]$. Then $X^{\prime},Y^{\prime}\in\mathcal{H}$.

\medskip
\textsc{Claim} 5: {\it $S_i=X^{\prime}+Y^{\prime}$.}

\medskip
{\it Proof}. Note that

\[
X^{\prime}=X\cap [0,n_{i+1}-1]\subseteq D_{i+1}\cap [0,n_{i+1}-1]=D_i
\]

\medskip
by Claim 4 and

\[
Y^{\prime}=Y\cap [0,n_{i+1}-1]\subseteq S_{i+1}\cap [0,n_{i+1}-1]=S_i=C_i+D_i.
\]

\medskip
This (together with (1)) implies that

\[
\max(X^{\prime}+Y^{\prime})=\max(X^{\prime})+\max(Y^{\prime})<3\max(D_i)\leq n_{i+1}.
\]

\medskip
Consequently,

\[
X^{\prime}+Y^{\prime}\subseteq (X+Y)\cap [0,n_{i+1}-1]=S_{i+1}\cap [0,n_{i+1}-1]=S_i.
\]

\medskip
Now let $v\in S_i$. Then $v\in S_{i+1}=X+Y$, and hence there are some $x\in X$ and $y\in Y$ such that $v=x+y$. Since $v\in [0,n_{i+1}-1]$, we have $\{x,y\}\subseteq [0,n_{i+1}-1]$, and thus $x\in X^{\prime}$ and $y\in Y^{\prime}$. Therefore, $v\in X^{\prime}+Y^{\prime}$.\qed

\medskip
Note that $C_i=C_{i+1}\cap [0,n_{i+1}-1]\subseteq Y^{\prime}$ by Claim 3. It follows by Claim 5 and the induction hypothesis that either $X^{\prime}=A_i$ and $Y^{\prime}=B_i$ or there is some $W\subseteq [0,i]$ such that

\[
X^{\prime}=\left\{\sum_E:E\subseteq W\right\}\textnormal{ and }Y^{\prime}=\left\{\sum_E:E\subseteq [0,i]\setminus W\right\}+C_i.
\]

\medskip
\textsc{Case} 1: {\it $X^{\prime}=A_i$ and $Y^{\prime}=B_i$.} First we show that $Y\subseteq B_{i+1}$. Assume that $Y\nsubseteq B_{i+1}$. Then there are some $J\subseteq [0,i+1]$ and $h\in J$ such that $m_h+\sum_J-1\in Y$. If $h\in [0,i]$, then $m_h\in A_i\subseteq X$, and hence $m_h+m_h+\sum_J-1\in S_{i+1}$, which contradicts ($\ast_4$). This implies that $h=i+1$. It follows by ($\ast_2$) that there are some $F,G\subseteq [0,i+1]$ such that

\[
[0,i+1]=F\cup G\textnormal{, }F\cap G\subseteq\{i+1\}\textnormal{, }i+1\in G\textnormal{, }\sum_F\in X\textnormal{ and }\sum_{\{i+1\}\setminus F}+\sum_G-1\in Y.
\]

\medskip
Since $\sum_F+m_{i+1}+\sum_J-1\in S_{i+1}$, we infer by ($\ast_4$) that $i+1\not\in F$. It follows that $m_{i+1}+\sum_G-1\in Y$. If $G\not=\{i+1\}$, then there is some $a\in G\cap [0,i]$, and thus $m_a\in A_i\subseteq X$ and $m_a+m_{i+1}+\sum_G-1\in S_{i+1}$, which contradicts ($\ast_4$). Therefore, $G=\{i+1\}$ and $F=[0,i]$. Consequently,

\[
\sum_{[0,i]}\in X\cap [0,n_{i+1}-1]=X^{\prime}=A_i,
\]

\medskip
which contradicts Claim 1. We conclude that $Y\subseteq B_{i+1}$.

Next we show that $X=A_{i+1}$. Since $2m_{i+1}-1>n_{i+1}+\max(D_i)=\max(B_{i+1})$, we have $2m_{i+1}-1\not\in B_{i+1}$. It follows that $2m_{i+1}-1\not\in Y$, and hence $m_{i+1}\in X$ by ($\ast_2$) (with $E=\{i+1\}$ and $r=i+1$). Therefore,

\[
A_{i+1}=A_i\cup\{m_{i+1}\}=X^{\prime}\cup\{m_{i+1}\}\subseteq X.
\]

\medskip
Now let $x\in X$. It follows from Claim 4 that there is some $E\subseteq [0,i+1]$ such that $x=\sum_E$. Since $\sum_{[0,i]}-1\in B_i=Y^{\prime}\subseteq Y$, we find that $\sum_E+\sum_{[0,i]}-1\in S_{i+1}$. It follows by ($\ast_3$) that $|E\cap [0,i]|\leq 1$. Consequently, there is some $j\in [0,i]$ such that $E\subseteq\{j,i+1\}$. Assume that $E=\{j,i+1\}$. There is some $j^{\prime}\in [0,i]\setminus\{j\}$. By ($\ast_2$) there are some $F,\overline{F},G,\overline{G}\subseteq [0,i+1]$ such that

\begin{align*}
&\{j,j^{\prime},i+1\}=F\cup G=\overline{F}\cup\overline{G}\textnormal{, }F\cap G\subseteq\{i+1\}\textnormal{, }\overline{F}\cap\overline{G}\subseteq\{j^{\prime}\}\textnormal{, }i+1\in G\textnormal{, }\\
&j^{\prime}\in\overline{G}\textnormal{, }\sum_F\in X\textnormal{, }\sum_{\overline{F}}\in X\textnormal{, }\sum_{\{i+1\}\setminus F}+\sum_G-1\in Y\textnormal{, }\sum_{\{j^{\prime}\}\setminus\overline{F}}+\sum_{\overline{G}}-1\in Y.
\end{align*}

\medskip
Since $Y\subseteq B_{i+1}$, we infer by Claim 1 that $\sum_G-1\in Y$, $\sum_{\overline{G}}-1\in Y$, $F\cap G=\{i+1\}$ and $\overline{F}\cap\overline{G}=\{j^{\prime}\}$. Note that

\[
\left\{\sum_E+\sum_G-1,\sum_F+\sum_{[0,i]}-1,\sum_{\overline{F}}+\sum_G-1,\sum_F+\sum_{\overline{G}}-1,\sum_{\overline{F}}+\sum_{[0,i]}-1\right\}\subseteq S_{i+1}.
\]

\medskip
It follows by ($\ast_3$) that

\[
|E\cap G|\leq 1\textnormal{, }|F\cap [0,i]|\leq 1\textnormal{, }|\overline{F}\cap G|\leq 1\textnormal{, }|F\cap\overline{G}|\leq 1\textnormal{ and }|\overline{F}\cap [0,i]|\leq 1.
\]

\medskip
Since $|E\cap G|\leq 1$, we have $j\not\in G$, and hence $j\in F$. Because $|F\cap [0,i]|\leq 1$, we infer that $j^{\prime}\not\in F$, and thus $j^{\prime}\in G$. Since $|\overline{F}\cap G|\leq 1$, this implies that $i+1\not\in\overline{F}$, and so $i+1\in\overline{G}$. Because $|F\cap\overline{G}|\leq 1$ and $j\in F$, we conclude that $j\not\in\overline{G}$. Consequently, $\{j,j^{\prime}\}\subseteq\overline{F}\cap [0,i]$, which contradicts $|\overline{F}\cap [0,i]|\leq 1$. Therefore, $E\subsetneq\{j,i+1\}$ and $x\in\{0,m_j,m_{i+1}\}\subseteq A_{i+1}$.

\medskip
Finally, we prove that $B_{i+1}\subseteq Y$. Let $y\in B_{i+1}$. Then there are some $r\in T\subseteq [0,i+1]$ with $y=\sum_T-1$. We infer by ($\ast_2$) that there are some $F,G\subseteq [0,i+1]$ such that

\[
T=F\cup G\textnormal{, }F\cap G\subseteq\{r\}\textnormal{, }r\in G\textnormal{, }\sum_F\in X\textnormal{ and }\sum_{\{r\}\setminus F}+\sum_G-1\in Y.
\]

\medskip
Since $Y\subseteq B_{i+1}$, we have $r\in F$ and $\sum_G-1\in Y$ by Claim 1. It follows from $\sum_F\in X=A_{i+1}$ and Claim 1 that $F=\{r\}$, and thus $G=T$. Consequently, $y=\sum_T-1\in Y$.

Therefore, $Y=B_{i+1}$, and hence $z=A_{i+1}B_{i+1}$.

\medskip
\textsc{Case} 2: {\it $X^{\prime}=\{\sum_E:E\subseteq W\}$ and $Y^{\prime}=\{\sum_E:E\subseteq [0,i]\setminus W\}+C_i$ for some $W\subseteq [0,i]$.} First we show that $X^{\prime}\subsetneq X$. Suppose that $X=X^{\prime}$. Set $Z^{\prime}=\{\sum_E:E\subseteq [0,i+1]\setminus W\}$. Then $Z^{\prime}\in\mathcal{H}$. It is sufficient to show that $Y=Z^{\prime}+C_{i+1}$. (Assume that we have already shown the last statement. Since $Y\in\mathcal{A}(\mathcal{H})$, it follows that $Z^{\prime}=\{0\}$, and hence $W=[0,i+1]$, a contradiction.)

\medskip
($\subseteq$) Let $y\in Y$. Then there are some $E\subseteq [0,i+1]$ and $r\in E$ such that either $y=\sum_E-1$ or $y=m_r+\sum_E-1$. First let $y=\sum_E-1$. It follows from ($\ast_3$) that $|W\cap E|\leq 1$. If $W\cap E=\emptyset$, then $E\setminus\{r\}\subseteq E\subseteq [0,i+1]\setminus W$, and thus

\[
y=\sum_{E\setminus\{r\}}+n_r\in Z^{\prime}+C_{i+1}.
\]

\medskip
If $s\in W\cap E$, then $E\setminus\{s\}\subseteq [0,i+1]\setminus W$, and hence $y=\sum_{E\setminus\{s\}}+n_s\in Z^{\prime}+C_{i+1}$. Now let $y=m_r+\sum_E-1$. It follows from ($\ast_4$) that $W\cap E=\emptyset$ and $E\subseteq [0,i+1]\setminus W$. We infer that $y=\sum_E+n_r\in Z^{\prime}+C_{i+1}$.

\medskip
($\supseteq$) Let $y\in Z^{\prime}+C_{i+1}$. Then there are some $E\subseteq [0,i+1]\setminus W$ and $r\in [0,i+1]$ such that $y=m_r+\sum_E-1$.

\medskip
\textsc{Case} a: {\it $r\not\in E\cup W$.} By ($\ast_1)$ there are some $F,G\subseteq [0,i+1]$ such that $E\cup\{r\}=F\cup G$, $F\cap G=\emptyset$, $\sum_F\in X$ and $\sum_G-1\in Y$. Note that $F=\emptyset$ by Claim 1 (since $X=X^{\prime}$), and hence $G=E\cup\{r\}$. This implies that $y=m_r+\sum_E-1=\sum_G-1\in Y$.

\medskip
\textsc{Case} b: {\it $r\in E\cup W$.} By ($\ast_2$) there are some $F,G\subseteq [0,i+1]$ such that $E\cup\{r\}=F\cup G$, $F\cap G\subseteq\{r\}$, $r\in G$, $\sum_F\in X$ and $\sum_{\{r\}\setminus F}+\sum_G-1\in Y$. Since $X=X^{\prime}$, we have $F\subseteq\{r\}$ by Claim 1, and hence $G=E\cup\{r\}$. Moreover, $r\not\in E$ if and only if $r\in W$ if and only if $m_r\in X$ if and only if $F=\{r\}$ by Claim 1 and ($\ast_4$). Observe that $y=m_r+\sum_E-1=\sum_{\{r\}\setminus F}+\sum_G-1\in Y$.

\medskip
Consequently, $X^{\prime}\subsetneq X$. By Claims 1 and 4 there is some $E^{\prime}\subseteq [0,i+1]$ with $i+1\in E^{\prime}$ and $\sum_{E^{\prime}}\in X$. Set $W^{\prime}=W\cup\{i+1\}$. It remains to show that

\[
Y=\left\{\sum_E:E\subseteq [0,i+1]\setminus W^{\prime}\right\}+C_{i+1}\textnormal{ and }X=\left\{\sum_E:E\subseteq W^{\prime}\right\}.
\]

\medskip
(Suppose that we have already shown the last statement. Since $Y\in\mathcal{A}(\mathcal{H})$, we infer that $W^{\prime}=[0,i+1]$, $X=D_{i+1}$ and $Y=C_{i+1}$. It follows by (1) that $z=C_{i+1}\prod_{j=0}^{i+1}\{0,m_j\}$ and we are done.)

\medskip
\textsc{Claim} 6: {\it $X\subseteq\{\sum_E:E\subseteq W^{\prime}\}$.}

\medskip
Let $x\in X$. Then there is some $E\subseteq [0,i+1]$ such that $x=\sum_E$ by Claim 4. Without restriction, we can assume that $W^{\prime}\subsetneq [0,i+1]$. There is some $j\in [0,i+1]\setminus W^{\prime}=[0,i]\setminus W$. Moreover,

\[
m_j+\sum_{[0,i]\setminus W}-1\in Y^{\prime}\subseteq Y.
\]

\medskip
We conclude that $\sum_E+m_j+\sum_{[0,i]\setminus W}-1\in S_{i+1}$, and thus $E\cap ([0,i+1]\setminus W^{\prime})=E\cap ([0,i]\setminus W)=\emptyset$ by ($\ast_4$). This implies that $E\subseteq W^{\prime}$.\qed

\medskip
\textsc{Claim} 7: {\it For all $E\subseteq [0,i+1]$ with $i+1\in E$, there are some $F,G\subseteq [0,i+1]$ such that $E=F\cup G$, $F\cap G=\{i+1\}$, $\sum_F\in X$ and $\sum_G-1\in Y$.}

\medskip
Let $E\subseteq [0,i+1]$ be such that $i+1\in E$. By ($\ast_2$), there are some $F,G\subseteq [0,i+1]$ with

\[
E=F\cup G\textnormal{, }F\cap G\subseteq\{i+1\}\textnormal{, }i+1\in G\textnormal{, }\sum_F\in X\textnormal{ and }\sum_{\{i+1\}\setminus F}+\sum_G-1\in Y.
\]

\medskip
Assume that $i+1\not\in F$. Then $m_{i+1}+\sum_G-1\in Y$, and so $\sum_{E^{\prime}}+m_{i+1}+\sum_G-1\in S_{i+1}$. By ($\ast_4$), we have that $i+1\in E^{\prime}\cap G=\emptyset$, a contradiction. Consequently, $F\cap G=\{i+1\}$ and $\sum_G-1\in Y$.\qed

\medskip
\textsc{Claim} 8: {\it $\{\sum_E:E\subseteq [0,i+1]\setminus W^{\prime}\}+C_{i+1}\subseteq Y$.}

\medskip
Since $\{\sum_E:E\subseteq [0,i+1]\setminus W^{\prime}\}+C_i=Y^{\prime}\subseteq Y$, it remains to prove that for each $E\subseteq [0,i]\setminus W$, we have $m_{i+1}+\sum_E-1\in Y$. Let $E\subseteq [0,i]\setminus W$. By Claim 7, there are some $F,G\subseteq [0,i+1]$ with

\[
E\cup\{i+1\}=F\cup G\textnormal{, }F\cap G=\{i+1\}\textnormal{, }\sum_F\in X\textnormal{ and }\sum_G-1\in Y.
\]

\medskip
Since $F\subseteq W^{\prime}$ by Claims 1 and 6, we find that $F=\{i+1\}$, $G=E\cup\{i+1\}$ and $m_{i+1}+\sum_E-1=\sum_G-1\in Y$.\qed

\medskip
\textsc{Claim} 9: {\it For all $J^{\prime}\subseteq J\subseteq [0,i+1]$ with $\sum_J\in X$, it follows that $\sum_{J^{\prime}}\in X$.}

\medskip
Let $J^{\prime}\subseteq J\subseteq [0,i+1]$ be such that $\sum_J\in X$. Then $J\subseteq W^{\prime}$ by Claims 1 and 6. If $i+1\not\in J^{\prime}$, then $J^{\prime}\subseteq W$ and $\sum_{J^{\prime}}\in X^{\prime}\subseteq X$. Now let $i+1\in J^{\prime}$. By Claim 7, there are some $\overline{F},\overline{G}\subseteq [0,i+1]$ such that

\[
J^{\prime}=\overline{F}\cup\overline{G}\textnormal{, }\overline{F}\cap\overline{G}=\{i+1\}\textnormal{, }\sum_{\overline{F}}\in X\textnormal{ and }\sum_{\overline{G}}-1\in Y.
\]

\medskip
Since $\sum_J+\sum_{\overline{G}}-1\in S_{i+1}$, it follows from ($\ast_3$) that $|J\cap\overline{G}|\leq 1$. Therefore, $\overline{G}=\{i+1\}$, $\overline{F}=J^{\prime}$ and $\sum_{J^{\prime}}\in X$.\qed

\medskip
By Claim 7, there are some $U,V\subseteq [0,i+1]$ such that $[0,i+1]=U\cup V$, $U\cap V=\{i+1\}$, $\sum_U\in X$ and $\sum_V-1\in Y$. Since $\sum_W\in X$, we see that $\sum_W+\sum_V-1\in S_{i+1}$, and thus $|W\cap V|\leq 1$ by ($\ast_3$). Assume that there is some $j\in W\cap V$. Then $U\subseteq W^{\prime}$ by Claims 1 and 6, and hence $U=W^{\prime}\setminus\{j\}$ and $V=([0,i+1]\setminus W)\cup\{j\}$.

\medskip
\textsc{Case} A: {\it $W\not=\{j\}$.} There is some $j^{\prime}\in W\setminus\{j\}$. By ($\ast_2$) there are some $\overline{F},\overline{G}\subseteq [0,i+1]$ such that

\[
\{j,j^{\prime},i+1\}=\overline{F}\cup\overline{G}\textnormal{, }\overline{F}\cap\overline{G}\subseteq\{j^{\prime}\}\textnormal{, }j^{\prime}\in\overline{G}\textnormal{, }\sum_{\overline{F}}\in X\textnormal{ and }\sum_{\{j^{\prime}\}\setminus\overline{F}}+\sum_{\overline{G}}-1\in Y.
\]

\medskip
Since $m_{j^{\prime}}\in X^{\prime}\subseteq X$, we know by ($\ast_4$) that $j^{\prime}\in\overline{F}$. Therefore, $\overline{F}\cap\overline{G}=\{j^{\prime}\}$ and $\sum_{\overline{G}}-1\in Y$. Since

\[
\sum_U+\sum_{\overline{G}}-1\in S_{i+1}\textnormal{ and }\sum_{\overline{F}}+\sum_V-1\in S_{i+1},
\]

\medskip
we infer by ($\ast_3$) that $|U\cap\overline{G}|\leq 1$ and $|\overline{F}\cap V|\leq 1$. Note that $U\cap\overline{G}=\{j^{\prime}\}$, and hence $i+1\not\in\overline{G}$. This implies that $i+1\in\overline{F}$. Consequently, $\overline{F}\cap V=\{i+1\}$ and $j\not\in\overline{F}$. Therefore, $j\in\overline{G}$ and $\overline{G}=\{j,j^{\prime}\}\subseteq W$. It follows that $\sum_{\overline{G}}\in X^{\prime}\subseteq X$, and thus $2\sum_{\overline{G}}-1\in S_{i+1}$. By ($\ast_3$), we have $|\overline{G}|\leq 1$, a contradiction.

\medskip
\textsc{Case} B: {\it $W=\{j\}$.} Observe that $U=\{i+1\}$, $V=[0,i+1]$, $m_{i+1}\in X$ and $\sum_{[0,i+1]}-1\in Y$. There is some $j^{\prime}\in [0,i]\setminus\{j\}$. By ($\ast_2$) there are some $\overline{F},\overline{G}\subseteq [0,i+1]$ such that

\[
\{j,j^{\prime},i+1\}=\overline{F}\cup\overline{G}\textnormal{, }\overline{F}\cap\overline{G}\subseteq\{j^{\prime}\}\textnormal{, }j^{\prime}\in\overline{G}\textnormal{, }\sum_{\overline{F}}\in X\textnormal{ and }\sum_{\{j^{\prime}\}\setminus\overline{F}}+\sum_{\overline{G}}-1\in Y.
\]

\medskip
Since

\[
\left\{\sum_{\overline{F}}+\sum_{[0,i+1]}-1,m_j+\sum_{\{j^{\prime}\}\setminus\overline{F}}+\sum_{\overline{G}}-1,m_{i+1}+\sum_{\{j^{\prime}\}\setminus\overline{F}}+\sum_{\overline{G}}-1\right\}\subseteq S_{i+1},
\]

\medskip
it follows from ($\ast_3$) that

\[
|\overline{F}|\leq 1\textnormal{, }|(\{j\}\cup (\{j^{\prime}\}\setminus\overline{F}))\cap\overline{G}|\leq 1\textnormal{ and }|(\{i+1\}\cup (\{j^{\prime}\}\setminus\overline{F}))\cap\overline{G}|\leq 1.
\]

\medskip
If $\overline{F}\subseteq\{i+1\}$, then

\[
(\{j\}\cup (\{j^{\prime}\}\setminus\overline{F}))\cap\overline{G}=\{j,j^{\prime}\},
\]

\medskip
a contradiction. If $\overline{F}\subseteq\{j\}$, then

\[
(\{i+1\}\cup (\{j^{\prime}\}\setminus\overline{F}))\cap\overline{G}=\{j^{\prime},i+1\},
\]

\medskip
a contradiction. Therefore, $\overline{F}=\{j^{\prime}\}$, and hence $m_{j^{\prime}}\in X^{\prime}$ and $j^{\prime}\in W$ by Claim 1, a contradiction.

\medskip
We infer that $W\cap V=\emptyset$, $U=W^{\prime}$ (by Claims 1 and 6), $V=[0,i+1]\setminus W$ and $\sum_{W^{\prime}}\in X$. By Claims 6 and 9, we deduce that $X=\{\sum_E:E\subseteq W^{\prime}\}$. We show that

\[
Y\subseteq\left\{\sum_E:E\subseteq [0,i+1]\setminus W^{\prime}\right\}+C_{i+1}.
\]

\medskip
(Then $Y=\{\sum_E:E\subseteq [0,i+1]\setminus W^{\prime}\}+C_{i+1}$ by Claim 8 and we are done.) Let $y\in Y$. There are some $T\subseteq [0,i+1]$ and $r\in T$ such that

\[
\textnormal{either }y=\sum_T-1\textnormal{ or }y=m_r+\sum_T-1.
\]

\medskip
Note that $\sum_{W^{\prime}}+y\in S_{i+1}$. If $y=m_r+\sum_T-1$, then $W^{\prime}\cap T=\emptyset$ by ($\ast_4$) and moreover, $y=\sum_T+n_r$, $T\subseteq [0,i+1]\setminus W^{\prime}$ and $n_r\in C_{i+1}$.

Now let $y=\sum_T-1$. It follows from ($\ast_3$) that $|W^{\prime}\cap T|\leq 1$. If $W^{\prime}\cap T=\emptyset$, then $y=\sum_{T\setminus\{r\}}+n_r$, $T\setminus\{r\}\subseteq [0,i+1]\setminus W^{\prime}$ and $n_r\in C_{i+1}$. If $s\in W^{\prime}\cap T$, then

\[
y=\sum_{T\setminus\{s\}}+n_s\textnormal{, }T\setminus\{s\}\subseteq [0,i+1]\setminus W^{\prime}\textnormal{ and }n_s\in C_{i+1}.
\]

\medskip
This completes the proof of $\mathsf{Z}(S_{i+1})=\{A_{i+1}B_{i+1},C_{i+1}\prod_{j=0}^{i+1}\{0,1+n_j\}\}$.

\medskip
Finally, let $j\in\mathbb{N}$. It is clear that $\mathsf{L}(S_j)=\{2,j+2\}$. Also note that $|A_j|\geq 3$, $|B_j|\geq 3$, $1\in A_j\setminus C_j$ and $m_1\in B_j\setminus C_j$. Therefore, for all $u,v\in\mathsf{Z}(S_j)$ with ${\rm gcd}(u,v)\not=1$, it follows that $u=v$.
\end{proof}

Next we determine the length set of certain elements that are not necessarily relatively cancellative. This will enable us later to show that $\{3,4,5\}$ is in the system of length sets of $\mathcal{P}_{{\rm fin},0}(\mathbb{N}_0)$.

\begin{proposition}\label{Proposition 3.6}
Let $\mathcal{H}=\mathcal{P}_{{\rm fin},0}(\mathbb{N}_0)$ and let $X\in\mathcal{H}$ and $n\in\mathbb{N}$ be such that $n>2\max(X)$. Set $\mathcal{M}=\{(A,C,D)\in\mathcal{H}^3:C$ and $D$ are relatively prime and $A+C=A+D=X\}$. Moreover, set $\mathcal{N}=\{A\in\mathcal{H}:(A,C,D)\in\mathcal{M}$ for some $C,D\in\mathcal{H}\}$. Then

\[
\mathsf{Z}(X+\{0,n\})=\bigcup_{(A,C,D)\in\mathcal{M}} (C\cup (n+D))\mathsf{Z}(A)\textnormal{,}\quad\mathsf{L}(X+\{0,n\})=1+\bigcup_{A\in\mathcal{N}}\mathsf{L}(A).
\]
\end{proposition}

\begin{proof}
First we prove

\medskip
\textsc{Claim} 1: {\it If $A,B\in\mathcal{H}$ are such that $\max(A)\leq\max(B)$ and $A+B=X+\{0,n\}$, then $B=C\cup (n+D)$ for some $C,D\in\mathcal{H}$ with $A+C=A+D=X$.}

\medskip
Let $A,B\in\mathcal{H}$ be such that $\max(A)\leq\max(B)$ and $A+B=X+\{0,n\}$. Observe that $2\max(A)\leq\max(A+B)=\max(X+\{0,n\})<2n$ and

\[
A+B=X\cup (n+X)\subseteq [0,\max(X)]\cup [n,n+\max(X)],
\]

\medskip
and thus $\max(A)\leq\max(X)$. There are some $a\in A$ and $b\in B$ such that $n=a+b$. If $b\in X$, then $n=a+b\leq 2\max(X)$, a contradiction. It follows that $b\not\in X$. Clearly, there is some $c\in X$ such that $b=c+n$, and hence $a=0$ and $n\in B$. Set $C=B\cap X$ and $D=\{x\in X:n+x\in B\}$. Then $C,D\in\mathcal{H}$ and $B=B\cap (X\cup (n+X))=C\cup (n+D)$. Also note that

\[
X\cup (n+X)=A+B=(A+C)\cup (A+(n+D))=(A+C)\cup (n+(A+D)).
\]

\medskip
Since $\max(X)<n$ and $\max(A+C)<n$, we infer that $A+C=X$ and $n+(A+D)=n+X$. Therefore, $A+D=X$.\qed

\medskip
\textsc{Claim} 2: {\it If $A,C,D\in\mathcal{H}$ are such that $A+C=A+D=X$, then $C$ and $D$ are relatively prime if and only if $C\cup (n+D)\in\mathcal{A}(\mathcal{H})$.}

\medskip
Let $A,C,D\in\mathcal{H}$ be such that $A+C=A+D=X$. Set $B=C\cup (n+D)$. Then clearly $B\in\mathcal{H}$, $B\not=\{0\}$ (since $n\in B$) and

\[
A+B=(A+C)\cup (A+(n+D))=X\cup (n+(A+D))=X\cup (n+X)=X+\{0,n\}.
\]

\medskip
Observe that $\max(A)\leq\max(X)<n\leq\max(B)$.

\medskip
First let $C$ and $D$ be relatively prime and let $Y,Z\in\mathcal{H}$ be such that $\max(Y)\leq\max(Z)$ and $B=Y+Z$. Then $(A+Y)+Z=X+\{0,n\}$. Note that $\max(Y)\leq\max(X)$ (since $Y\subseteq X+\{0,n\}\subseteq [0,\max(X)]\cup [n,n+\max(X)]$ and $\max(B)<2n$). Consequently, $\max(A+Y)\leq 2\max(X)<n\leq\max(Z)$, and thus $Z=C^{\prime}\cup (n+D^{\prime})$ for some $C^{\prime},D^{\prime}\in\mathcal{H}$ with $(A+Y)+C^{\prime}=(A+Y)+D^{\prime}=X$ by Claim 1. Since

\[
C\cup (n+D)=Y+Z=(Y+C^{\prime})\cup (n+(Y+D^{\prime})),
\]

\medskip
$\max(C)<n$ and $\max(Y+C^{\prime})<n$, we find that $C=Y+C^{\prime}$ and $D=Y+D^{\prime}$. Consequently, $Y=\{0\}$ (since $C$ and $D$ are relatively prime). We infer that $B\in\mathcal{A}(\mathcal{H})$.

\medskip
Now let $B\in\mathcal{A}(\mathcal{H})$ and let $A^{\prime},C^{\prime},D^{\prime}\in\mathcal{H}$ be such that $C=A^{\prime}+C^{\prime}$ and $D=A^{\prime}+D^{\prime}$. Then

\[
B=A^{\prime}+(C^{\prime}\cup (n+D^{\prime}))
\]

\medskip
and since $C^{\prime}\cup (n+D^{\prime})\not=\{0\}$, we conclude that $A^{\prime}=\{0\}$. Therefore, $C$ and $D$ are relatively prime.\qed

\medskip
First we show that

\[
\mathsf{Z}(X+\{0,n\})=\bigcup_{(A,C,D)\in\mathcal{M}} (C\cup (n+D))\mathsf{Z}(A).
\]

\medskip
If $(A,C,D)\in\mathcal{M}$, then $C\cup (n+D)\in\mathcal{A}(\mathcal{H})$ by Claim 2 and

\[
(C\cup (n+D))+A=(A+C)\cup (n+(A+D))=X+\{0,n\},
\]

\medskip
and thus $(C\cup (n+D))\mathsf{Z}(A)\subseteq\mathsf{Z}(X+\{0,n\})$. Now let $z\in\mathsf{Z}(X+\{0,n\})$. Then there are some $m\in\mathbb{N}$ and atoms $(X_i)_{i=1}^m$ of $\mathcal{H}$ such that $z=\prod_{i=1}^m X_i$ and $\max(X_i)\leq\max(X_{i+1})$ for each $i\in [1,m-1]$. Observe that

\[
\sum_{i=1}^m X_i=X+\{0,n\}\textnormal{ and }\sum_{i=1}^m\max(X_i)=\max(X)+n.
\]

\medskip
Set

\[
r=\max\left(\left\{j\in [0,m]:\sum_{i=1}^j\max(X_i)\leq\max(X)\right\}\right)\textnormal{, }A=\sum_{i=1}^r X_i\textnormal{ and }B=\sum_{i=r+1}^m X_i.
\]

\medskip
Then $A,B\in\mathcal{H}$ and $A+B=X+\{0,n\}$. As $\max(A)+\max(B)=\max(X)+n$ and $\max(A)\leq\max(X)$, we have $\max(A)\leq\max(X)<n\leq\max(B)$. If $\max(X_{r+1})\leq\max(X)$, then

\[
\max(X)<\sum_{i=1}^{r+1}\max(X_i)\leq 2\max(X)<n\textnormal{ and }\sum_{i=1}^{r+1}\max(X_i)\in X+\{0,n\},
\]

\medskip
a contradiction. This implies that $\max(X_i)\geq n$ for each $i\in [r+1,m]$. If $m>r+1$, then $\max(X_{r+1})+\max(X_m)\in X+\{0,n\}$, and thus

\[
2n\leq\max(X_{r+1})+\max(X_m)\leq\max(X)+n,
\]

\medskip
a contradiction. Therefore, $m=r+1$ and $B=X_m\in\mathcal{A}(\mathcal{H})$. If follows from Claims 1 and 2 that there are some $C,D\in\mathcal{H}$ such that $(A,C,D)\in\mathcal{M}$ and $B=C\cup (n+D)$. Moreover,

\[
z=(C\cup (n+D))\prod_{i=1}^r X_i\in (C\cup (n+D))\mathsf{Z}(A).
\]

\medskip
Finally,

\begin{align*}
\mathsf{L}(X+\{0,n\})&=\{|z|:z\in\mathsf{Z}(X+\{0,n\})\}\\
&=\bigcup_{(A,C,D)\in\mathcal{M}}\{|(C\cup (n+D))z|:z\in\mathsf{Z}(A)\}\\
&=\bigcup_{(A,C,D)\in\mathcal{M}}\{1+|z|:z\in\mathsf{Z}(A)\}\\
&=1+\bigcup_{(A,C,D)\in\mathcal{M}}\mathsf{L}(A)\\
&=1+\bigcup_{A\in\mathcal{N}}\mathsf{L}(A)
\end{align*}

\medskip
by Claim 2.
\end{proof}

\begin{example}\label{Example 3.7}
Let $\mathcal{H}=\mathcal{P}_{{\rm fin},0}(\mathbb{N}_0)$, let $X=\{0,1,4,5,10,11,12,14,15,16,19,20,21,22,25,26,29,30\}$ and let $Y=X+\{0,61\}$. Then $\mathsf{L}(X)=\{2,3,4\}$, $\mathsf{L}(Y)=\{3,4,5\}$ and $X$ and $Y$ are not relatively cancellative.
\end{example}

\begin{proof}
Let $A,B\in\mathcal{H}$ be such that $0<\max(A)\leq\max(B)$ and $A+B=X$. Note that

\[
(\max(A),\max(B))\in\{(1,29),(4,26),(5,25),(10,20),(11,19),(14,16),(15,15)\}.
\]

\medskip
\textsc{Case} 1: {\it $\max(A)=1$.} Then $A=\{0,1\}$ and it follows that

\[
\{0,4,10,11,14,15,19,21,25,29\}\subseteq B\subseteq\{0,4,10,11,14,15,19,21,25,29\}\cup\{20\}.
\]

\medskip
Also note that $B\not\in\mathcal{A}(\mathcal{H})$.

\medskip
\textsc{Case} 2: {\it $\max(A)=4$.} Since $27\not\in X$, we obtain $A=\{0,4\}$ and $B=\{0,1,10,11,12,15,16,21,22,25,26\}$. Observe that $B\not\in\mathcal{A}(\mathcal{H})$.

\medskip
\textsc{Case} 3: {\it $\max(A)=5$.} Since $6\not\in X$ and $1\in A\cup B$, it follows that $1\in A$. Moreover, $9\not\in X$ and $4\in A\cup B$, and thus $4\in A$. Therefore, $A=\{0,1,4,5\}$ and $B=\{0,10,11,15,21,25\}$. Furthermore, $\{A,B\}\cap\mathcal{A}(\mathcal{H})=\emptyset$.

\medskip
\textsc{Case} 4: {\it $\max(A)=10$.} Observe that $19\in B$ (since $29\in X$). This implies that $\{4,5\}\cap A=\emptyset$. Moreover, $16\in B$ (since $26\in X$), and so $1\not\in A$. We infer that $A=\{0,10\}$ and

\[
\{0,1,4,5,11,12,15,16,19,20\}\subseteq B\subseteq\{0,1,4,5,11,12,15,16,19,20\}\cup\{10\}.
\]

\medskip
Note that $B\not\in\mathcal{A}(\mathcal{H})$.

\medskip
\textsc{Case} 5: {\it $\max(A)=11$.} Clearly, $\{4,5\}\cap A=\emptyset$ (since $19\in B$). Since $29\in X$, we have $10\in A$. Consequently,

\[
\{0,10,11\}\subseteq A\subseteq\{0,1,10,11\}.
\]

\medskip
If $A=\{0,10,11\}$, then $B=\{0,1,4,5,10,11,15,19\}$.

If $A=\{0,1,10,11\}$, then

\medskip
\[
\{0,4,11,15,19\}\subseteq B\subseteq\{0,4,10,11,15,19\}.
\]

\medskip
We emphasize that $\{0,1,10,11\}\not\in\mathcal{A}(\mathcal{H})$ and $\{0,4,11,15,19\}\not\in\mathcal{A}(\mathcal{H})$.

\medskip
\textsc{Case} 6: {\it $\max(A)=14$.} Then $\{1,11,12\}\cap A=\emptyset$ (since $16\in B$) and $\{4,10,14\}\cap B=\emptyset$. This implies that $1\in B$ and $4\in A$, and hence $5\not\in A\cup B$. In particular, $10\in A$ (since $10\in X$). Observe that $A=\{0,4,10,14\}\not\in\mathcal{A}(\mathcal{H})$ and $B=\{0,1,11,12,15,16\}\not\in\mathcal{A}(\mathcal{H})$.

\medskip
\textsc{Case} 7: {\it $\max(A)=15$.} Then $12\not\in A\cup B$. First let $1\in B$. Then $\{1,5\}\cap A=\emptyset$. Since $22\in X$, we have $11\in A\cap B$. Note that $\{4,10,14\}\subseteq A\cup B$ (since $\{4,10,29\}\subseteq X$), and thus $10\not\in A\cap B$ (since $14\in A\cup B$). We infer that $5\in B$ (since $20\in X$), and hence $4\not\in A$ and $4\in B$. Therefore, $14\in B$, and thus $10\in B$. We infer that $\{10,14\}\cap A=\emptyset$. Observe that $A=\{0,11,15\}$ and $B=\{0,1,4,5,10,11,14,15\}\not\in\mathcal{A}(\mathcal{H})$. If $1\not\in B$, then it follows by analogy that $A=\{0,1,4,5,10,11,14,15\}\not\in\mathcal{A}(\mathcal{H})$ and $B=\{0,11,15\}$.

\medskip
Observe that $\{0,4,10,11,15,19\}\in\mathcal{A}(\mathcal{H})$ (for if $C,D\in\mathcal{H}$ are such that $0<\max(C)\leq\max(D)$ and $C+D=\{0,4,10,11,15,19\}$, then $\max(C)=4$, $10\in D$ and $14\in C+D$, a contradiction). We infer in a similar way that $\{0,1,4,5,10,11,15,19\}\in\mathcal{A}(\mathcal{H})$. Therefore,

\begin{align*}
&\{U\in\mathcal{A}(\mathcal{H}):U\mid_{\mathcal{H}} X\}=\\
&\{\{0,1\},\{0,4\},\{0,10\},\{0,10,11\},\{0,11,15\},\{0,4,10,11,15,19\},\{0,1,4,5,10,11,15,19\}\}
\end{align*}

\medskip
as shown in the case analysis. It follows that

\begin{align*}
&\mathsf{Z}(X)=\\
&\{\{0,1\}\{0,4\}\{0,10\}\{0,11,15\},\{0,1\}\{0,10\}\{0,4,10,11,15,19\},\{0,10,11\}\{0,1,4,5,10,11,15,19\}\}.
\end{align*}

\medskip
Obviously, $\mathsf{L}(X)=\{2,3,4\}$ and

\[
\{U\in\mathcal{H}:U+C=U+D=X\textnormal{ for some relatively prime }C,D\in\mathcal{H}\}=\{\{0,1,10,11\},X\}.
\]

\medskip
Therefore, for $Y=X+\{0,61\}$,

\[
\mathsf{L}(Y)=1+(\mathsf{L}(\{0,1,10,11\})\cup\mathsf{L}(X))=\{3,4,5\}
\]

\medskip
by Proposition~\ref{Proposition 3.6}. From the description of $\mathsf{Z}(X)$, it is clear that $X$ and $Y$ are not relatively cancellative.
\end{proof}

Now we are prepared to give the proof of our main result.

\begin{proof}[Proof of Theorem~\ref{Theorem 1.1}]
(1) Recall that $\mathcal{H}^*$ is the submonoid of $\mathcal{P}_{{\rm fin}}(\mathbb{N}_0)$ generated by

\[
\big\{\{1\},\{2,n\}:n\in\mathbb{N}_{\geq 3}\big\}.
\]

Set $\mathcal{H}^{\prime}=\mathcal{P}_{{\rm fin},0}(\mathbb{N}_0)$ and let $n\in\mathbb{N}_{\geq 3}$. It follows from Lemma~\ref{Lemma 3.4} and Theorem~\ref{Theorem 3.5}(2) that there is some relatively cancellative $Y\in\mathcal{H}^{\prime}$ such that $\mathsf{L}(Y)=\{2,n\}$. Moreover, $\{0,1\}\in\mathcal{H}^{\prime}$ is relatively cancellative and $\mathsf{L}(\{0,1\})=\{1\}$. We infer by induction from Corollary~\ref{Corollary 3.2} that $\mathcal{H}^*\subseteq\mathcal{L}(\mathcal{H}^{\prime})$. Finally, $\mathcal{L}(\mathcal{H}^{\prime})\subseteq\mathcal{L}(\mathcal{H})$ by \cite[Theorem 4.11]{FaTr18}.

\bigskip
(2) Let $k\in\mathbb{N}_{\geq 2}$. If $k\in\{2,3\}$, then $[k,k+2]\in\mathcal{L}(\mathcal{H})$ by Example~\ref{Example 3.7} and \cite[Theorem 4.11]{FaTr18}. If $k\geq 4$, then

\[
[k,k+2]=\{k-4\}+\{2,3\}+\{2,3\}\in\mathcal{H}^*\subseteq\mathcal{L}(\mathcal{H})
\]

\medskip
by (1).

\bigskip
(3) It follows from \cite[Proposition 3.2]{FaTr18} that $\mathcal{H}$ is atomic. Moreover,

\[
\{\{m,n\}:m,n\in\mathbb{N}_{\geq 2},n\geq m\}\subseteq\mathcal{H}^*\subseteq\mathcal{L}(\mathcal{H})
\]

\medskip
by (1). Let $r\in\mathbb{Q}_{\geq 1}$. Then there are some $n,m\in\mathbb{N}_{\geq 2}$ and $a\in\mathcal{H}$ such that $n\geq m$, $r=\frac{n}{m}$ and $\mathsf{L}(a)=\{m,n\}$. Observe that $\rho(a)=r$. Consequently, $\mathcal{H}$ is fully elastic and $\rho(\mathcal{H})=\infty$.
\end{proof}

\bigskip
\noindent {\bf ACKNOWLEDGEMENTS.} We want to thank A. Geroldinger and the referee for helpful remarks and suggestions.

\end{document}